\documentclass{amsart}
\usepackage{graphicx}
\usepackage{amsthm}
\usepackage{amsmath}
\usepackage{amssymb}
\usepackage{color}
\newtheorem{thm}{Theorem}[section]

\theoremstyle{definition}

\theoremstyle{remark}
\newtheorem{remark}[thm]{Remark}

\newtheorem{theorem}{Theorem}

\newtheorem{lemma}{Lemma}
\newtheorem{proposition}{Proposition}
\newtheorem{definition}{Definition}
\newtheorem{corollary}{Corollary}

\newtheorem{example}{Example}

\newcommand{\RR}{\mathbb R}

\numberwithin{equation}{section}

\begin{document}
\title[Weak phase retrieval ]{Classifying weak phase retrieval }
\author{P. G. Casazza}
\address{Department of Mathematics, University of Missouri, Columbia, USA.}
\email{casazzapeter40@gmail.com}
\author{F. Akrami}
\address{Department of Mathematics, University of Maragheh, Maragheh, Iran.}
\email{fateh.akrami@gmail.com}
\dedicatory{}

\thanks{The first author was supported by NSF DMS 1609760}

\dedicatory{}
\subjclass[2010]{42C15, 42C40.}

\keywords{Real Hilbert frames, Full spark, Phase retrieval, Weak phase retrieval.}

\begin{abstract}
We will give several surprising equivalences and consequences of weak phase retrieval. These results give a complete understanding
of the difference between weak phase retrieval and phase retrieval.
We also answer two longstanding open problems on weak phase retrieval:  (1) We show that the families of weak phase retrievable frames $\{x_{i}\}_{i=1}^{m}$ in $\mathbb{R}^n$ are not dense in the family of $m$-element sets of
vectors in $\mathbb{R}^n$ for  all $m\ge 2n-2$; (2) We show that any frame $\{x_i\}_{i=1}^{2n-2}$ 
containing one or more canonical basis vectors in $\mathbb{R}^n$ cannot do weak phase retrieval.
We provide numerous examples to show that the obtained results are best possible.
\end{abstract}

\subjclass[2010]{42C15, 42C40.}

\maketitle

\section{Introduction}\label{s:intro}
The concept of frames in a separable Hilbert space was originally introduced by Duffin and Schaeffer in the context of non-harmonic Fourier series \cite{DS1952}.
Frames are a more flexible tool than bases because of  the redundancy property that make them more applicable than bases.
 Phase retrieval is an old problem of recovering a signal from the absolute value of linear measurement coefficients called intensity measurements.
 Phase retrieval and norm retrieval have become very active areas of research in applied mathematics, computer science, engineering, and more today. Phase retrieval has been defined for both vectors and subspaces (projections) in all separable Hilbert spaces, (e.g., \cite{CCJW14}, \cite{BCE2006}, \cite{BCCHT18}, \cite{ACCHTTX17}, \cite{CCD16}, \cite{CCPW13} and \cite{CGJT17}).
\par
The concept of weak phase retrieval weakened the notion of phase retrieval and it has been first defined for vectors in (\cite{ACNT2015} and \cite{ACGJT2016}). 
 The rest of the paper is organized as follows: In Section 2, we give the basic
definitions and certain preliminary results to be used in the paper. Weak
phase retrieval by vectors is introduced in section 3. 
In section 4 we show that any family of vectors $\{x_i\}_{i=1}^{2n-2}$ doing weak phase retrieval cannot contain a unit vector. 
In section 5, we show that the weak phase retrievable frames are not dense in all frames. And in section 6
we give several surprising equivalences and consequences of weak phase retrieval. These results give a complete understanding
of the difference between weak phase retrieval and phase retrieval.

\section{preliminaries} \label{s:pre}
First we give the background material needed for the paper. Let $\mathbb{H}$ be a finite or infinite dimensional real Hilbert space and B$(\mathbb{H})$ the class of all bounded linear operators defined on $\mathbb{H}$. The natural numbers and real numbers are denoted by $``\mathbb{N}"$ and $``\mathbb{R}"$, respectively. We  use $[m]$ instead of the set $\{1,2,3,\dots,m \}$ and use $[\{x_i\}_{i \in I}]$ instead of $span\{x_i\}_{i \in I}$, where $I$ is a finite or countable subset of $\mathbb{N}$. We denote by $\mathbb{R}^n$ a $n$ dimensional real  Hilbert space.
 We start with the definition of a real Hilbert space frame.
\begin{definition} \label{D:frame}
 A family of vectors $\{x_i\}_{i\in I}$ in a finite or infinite dimensional separable real Hilbert space $\mathbb {H}$ is a \textbf{frame} if there are constants $0<A \leq B< \infty $ so that
$$ A\|x\|^2 \leq \sum_{i\in I}{|\langle x,x_i \rangle|^2}\leq B\|x\|^2, \quad \mbox{for all} \quad f\in \mathbb{H}.$$
The constants $A$  and $B$ are called the lower and upper frame bounds for $\{x_i\}_{i\in I}$, respectively. If an upper frame bound exists, then $\{x_i\}_{i\in I}$ is called a {\bf B-Bessel seqiemce} or simply
{\bf Bessel} when the constant is implicit. If $A=B$, it is called an {\bf A-tight frame} and
 in case $ A=B=1$, it is called a {\bf Parseval~frame}.  The values $\{\langle x,x_i \rangle\}_{i=1}^{\infty}$ are called the frame coefficients of the vector $x \in \mathbb{H}$.
\end{definition}

It is immediate that a frame must span the space.
We will need to work with Riesz sequences.
\begin{definition} \label{D:Riesz sequence}
A family $X = \{x_i\}_{i \in I}$ in a finite or infinite dimensional real Hilbert space $\mathbb{H}$
is a \textbf{Riesz sequence} if there are constants $0<A \leq B< \infty $ satisfying
$$A\sum_{i \in I}|c_i|^2 \leq \|\sum_{i\in I}{c_i x_i}\|^2\leq B\sum_{i \in I}|c_i|^2$$
 for all sequences of scalars $\{c_i\}_{i \in I}$.
If it is complete in $\mathbb{H}$, we call $X$ a \textbf{Riesz basis}.
\end{definition}
For an introduction to frame theory we recommend \cite{CK2013, Ch2003}.\\
Throughout the paper the orthogonal projection or simply projection will be a self-adjoint positive projection and $\{e_i\}_{i=1}^{\infty}$ will be used to denote the canonical basis for the real space $\mathbb{R}^n$, i.e., a basis for which
$$\langle e_i,e_j \rangle=\delta_{i,j}=
\begin{cases}
1 \quad  if \ i=j,  \\
0 \quad  if \ i\ne j.
\end{cases}$$

\begin{definition}\label{D:phase(norm) retrieval by vectors}
\textrm A family of vectors $\{x_{i}\}_{i\in I}$ in a real Hilbert space $\mathbb {H}$ does \textbf {phase (norm) retrieval} if whenever $ x, y \in \mathbb {H}$, satisfy
$$|\langle x,x_{i}\rangle |=|\langle y,x_{i}\rangle| \quad \mbox{ for all } i\in I,$$
then $x=\pm y$ \, $(\|x\|=\|y\|). $
\end{definition}
Phase retrieval was introduced in reference \cite{BCE2006}. See reference \cite{ACHR2021} for an introduction to norm retrieval.

Note that if $\{x_i\}_{i\in I}$ does phase (norm) retrieval, then so does $\{a_ix_i\}_{i\in I}$ for any $0< a_i< \infty$ for all $i\in I$.
But in the case where $|I|=\infty$, we have to be careful to maintain frame bounds. This always works if $0<\inf_{i\in I}a_i \le sup_{i\in I}a_i < \infty$.
But this is not necessary in general \cite{ACHR2021}.
 The complement property is an essential issue here.

\begin{definition}\label{D:complement pro in infinite case}
\textrm A family of vectors $\{x_{i}\}_{i\in I}$ in a finite or infinite dimensional real Hilbert space $\mathbb{H}$ has the \textbf{complement property}\, \textrm  if for any subset $J\subset I$, \\
$$either\ \  \overline{span}\{x_{i}\}_{i\in J}=\mathbb{H} \textrm \quad or \quad \ \  \overline{span}\{x_{i}\}_{i\in J^c}=\mathbb{H}. $$
\end{definition}
Fundamental to this area is the following for which the finite dimensional case appeared in \cite{CCPW13}.
\begin{theorem}\label{T:p34}
A family of vectors $\{x_{i}\}_{i\in I} $ does phase retrieval in $\mathbb{R}^n$ if and only if it has the complement property.
\end{theorem}

We recall:

\begin{definition}\label{D:full spark}
\textrm A family of vectors $\{x_{i}\}_{i=1}^{m}$ \textrm in $\mathbb{R}^n$ is full spark if for every $I \subset [m] \, \mbox with \, |I|=n$ \textrm, $\{x_{i}\}_{i \in I}$ \textrm
is linearly independent.\end{definition}
\begin{corollary}\label{T:xx11}
\textrm If $\{x_{i}\}_{i=1}^{m} $  \textrm does phase retrieval in $\mathbb {R}^n$, then $m \geq2n-1$. If $m=2n-1$, $\{x_{i}\}_{i=1}^{m} $ does phase retrieval if and only if it is full spark.
\end{corollary}

We rely heavily on a significant result from \cite{CAR2021}. In that theorem they assume $m=2n-2$, but this is clearly not necessary.
\begin{theorem}\label{T52}
If $\{x_i\}_{i=1}^{m}$ does weak phase retrieval in $\RR^n$ then for every $I\subset [m]$,
if $x,y\in \RR^n$, $\|x\|=\|y\|=1$ and $x\perp span\{x_i\}_{i\in I}$ and $y\perp \{x_i\}_{i\in I^c}$ then 
$x+y$ and $x-y$ are disjointly supported.
\end{theorem}
\begin{remark}\label{s18}
The above theorem may fail if $\|x\| \ne \|y\|$. For example, consider the weak phase retrievable frame in $\RR^3$:
\[ \begin{bmatrix}
 1&1&1\\
 -1&1&1\\
 1&-1&1\\
 1&1&-1
 \end{bmatrix}\]
Also, $x=(0,1,-1)$ is perpendicular to rows 1 and 2 and $y=(0,\frac{1}{2},\frac{1}{2})$ is orthogonal to rows 2 and 3.
But $x+y=(0,\frac{3}{2},\frac{1}{2})$ and $x-y=(0,\frac{-1}{2},\frac{-3}{2})$ and these are not disjointly supported.
But if we let them have the same norm we get $x=(0,1,-1)$ and $y=(0,1,1)$ so $x+y=(0,1,0)$ and $x-y=(0,0,1)$ and these are
disjointly supported.
\end{remark}

\section{Weak phase retrieval}\label{s:Weak phase retrieval by vectors and projections }
The notion of ``Weak phase retrieval by vectors'' in $\mathbb{R}^n$ was introduced in \cite{ACNT2015} and was
developed further in \cite{ACGJT2016}. One limitation of current methods used for retrieving
the phase of a signal is computing power. Recall that a generic family of $(2n-1)$-
vectors in $\mathbb{R}^n$ satisfies phaseless reconstruction, however no set of $(2n-2)$-vectors
can (See \cite{ACGJT2016} for details). By generic we are referring to an open dense set in the set
of $(2n-1)$-element frames in $\mathbb{R}^n$. 

\begin{definition}\label{D:Weakly have the same phase}
Two vectors  $x=(a_1,a_2,\dots,a_n)$ and $y=(b_1,b_2,\dots,b_n)$ in $\mathbb{R}^n$
{\bf weakly have the same phase} if there is a $|\theta|=1$ so that
$phase(a_i)=\theta phase(b_i)$ for all $i \in [n]$, for which $a_i \ne 0 \ne b_i$.\\
If $\theta=1$, we say $x$ and $y$ weakly have the same signs and if $\theta=-1$, they weakly have the opposite signs.
\end{definition}
Therefore with above definition the zero vector in $\RR^n$  weakly has the same phase with all vectors in $\RR^n$.
For $x\in \RR$, $sgn(x)=1$ if $x>0$ and $sgn(x)=-1$ if $x<0$.
\begin{definition}\label{D:Weak phase retrieval by vectors}
A family of vectors $\{x_i\}_{i=1}^m$ does {\bf weak phase retrieval} in $\mathbb{R}^n$ if for any $x=(a_1,a_2,\dots,a_n)$ and $y=(b_1,b_2,\dots,b_n)$ in $\mathbb{R}^n$ with $|\langle x,x_i \rangle |=|\langle y,x_i \rangle |$ for all $i \in [m]$, then $x$ and $y$ weakly have the same phase.
\end{definition}
  A fundamental result here is
\begin{proposition}  \label{D:pp}\cite{ACNT2015}
Let $x=(a_1,a_2,\dots,a_n)$ and $y=(b_1,b_2,\dots,b_n)$ in $\mathbb{R}^n$. The following are equivalent:\\
(1) We have $sgn(a_i a_j)=sgn(b_i b_j)$,\: for  all $1 \leq i \ne j \leq n$ \\
(2) Either $x,y$ have weakly the same sign or they have the opposite signs.
\end{proposition}
It is clear that if $\{x_i\}_{i=1}^m$ does weak phase retrieval  in $\RR^n$, then
$\{c_ix_i\}_{i=1}^m$ does  weak phase retrieval as long as $c_i>0$ for all $i=1,2,\ldots,m$.

The following appears in \cite{ACGJT2016}.
\begin{theorem}
If $X=\{x_i\}_{i=1}^m$ does weak phase retrieval in $\mathbb{R}^n$, then $m \geq 2n-2$.
\end{theorem}
Finally, we have:
\begin{theorem}\label{D:fa66}\cite{ACGJT2016}
 If a frame $X=\{x_i\}_{i=1}^{2n-2}$ does weak phase retrieval in $\mathbb{R}^n$, then $X$ is a full spark frame.
\end{theorem}
Clearly the converse of above theorem is not hold, for example $\{(1,0),(0,1)\}$ is full spark frame that fails weak phase retrieval in $\RR^2$.\\

If $\{x_i\}_{i\in I}$ does phase retrieval and $R$ is an invertible operator on the space then $\{Rx_i\}_{i\in I}$ does phase retrieval. 
This follows easily since $|\langle x,Rx_i\rangle|=|\langle y,Rx_i\rangle|$ implies $|\langle R^*x,x_i\rangle|=|\langle R^*y,x_i\rangle|$,
and so $R^*x=\theta R^*y$ for $|\theta|=1$. Since $R$ is invertible, $x=\theta y$.
This result fails badly for
weak phase retrieval. For example, let $e_1=(1,0),\ e_2=(0,1),\ x_1=(\frac{1}{\sqrt{2}},\frac{1}{\sqrt{2}},\ x_2=(\frac{1}{\sqrt{2}},\frac{-1}{\sqrt{2}})$ in $\RR^2$.
Then $\{e_1,e_2\}$ fails weak phase retrieval, $\{x_1,x_2\}$ does weak phase retrieval and $Ue_i=x_i$ is a unitary operator.

\section{Frames Containing Unit Vectors}

\begin{theorem}\label{A1}
Any frame $\{x_i\}_{i=1}^{2n-2}$ whith one or more canonical basis vectors in $\mathbb{R}^n$ cannot do weak phase retrieval.
\end{theorem}

\begin{proof} We proceed by way of contradiction.
Recall that $\{x_i\}_{i=1}^{2n-2}$ must be full spark. Let $\{e_i\}_{i=1}^n$ be the canonical orthonormal basis
of $\RR^n$. 
Assume $I\subset \{1,2,\ldots,2n-2\}$ with $|I|=n-1$ and assume
$x=(a_1,a_2,\ldots,a_n)$, $y=(b_1,b_2,\ldots ,b_n)$ with $\|x\|=\|y\|=1$ and $x\perp X=span\{x_i\}_{i\in I}$ and
$y\perp span\{x_i\}_{i=n}^{2n-2}$. 
After reindexing $\{e_i\}_{i=1}^n$ and $\{x_i\}_{i=1}^{2n-2}\}$, we assume $x_1=e_1$, $I=\{1,2,\ldots,n-1$ and
$I^c=\{n,n+1,\ldots,2n-2\}$. Since $\langle x,x_1\rangle = a_1=0$, by Theorem \ref{T52}, $b_1=0$.
 Let $P$ be the projection on $span\{e_i\}_{i=2}^{n}$. So $\{Px_i\}_{i=n}^{2n-2}$ is
$(n-1)$-vectors in an $(n-1)$-dimensional space and $y$ is orthogonal to all these vectors. So there exist
$\{c_i\}_{i=n}^{2n-2}$ not all zero so that 
\[ \sum_{i=n}^{2n-2}c_iPx_i=0\mbox{ and so }\sum_{i=n}^{2n-1}c_ix_i(1)x_1-\sum_{i=n}^{2n-2}c_ix_i=0.\]
That is, our vectors are not full spark, a contradiction.
\end{proof}

\begin{remark}
The fact that there are $(2n-2)$ vectors in the theorem is critical. For example, $e_1,e_2,e_1+e_2$ is full spark
in $\RR^2$, so it does phase retrieval - and hence weak phase retrieval - despite the fact that it contains both
basis vectors.
\end{remark}
The converse of Theorem \ref{A1} is not true in general.
\begin{example}
Consider the full spark frame
$X=\{(1,2,3),(0,1,0),(0,-2,3),(1,-2,-3)\}$ in $\RR^3$. Every set of its two same coordinates,
\[\{(1,2),(0,1),(0,-2),(1,-2)\},\ \{(1,3),(0,0),(0,3),(1,-3)\},\mbox{ and }\]
\[\{(2,3),(1,0),(-2,3),(-2,-3)\}\]
do weak phase retrieval in $\RR^2$, but by Theorem  \ref{A1}, $X$ cannot do weak phase retrieval in $\mathbb{R}^3$.
\end{example}

\section{Weak Phase Retrievable Frames are not Dense in all Frames}

If $m\ge 2n-1$ and $\{x_i\}_{i=1}^m$ is full spark then it has complement property and hence does phase retrieval.
Since the full spark frames are dense in all frames, it follows that the frames doing phase retrieval are dense in all
frames with $\ge 2n-1$ vectors.
We will now show that this result fails for weak phase retrievable frames.
The easiest way to get very general frames failing weak phase retrieval is:
\vskip10pt
Choose $x,y\in \RR^n$ so that $x+y,x-y$ do not have the same or opposite signs.
Let $X_1=x^{\perp}$ and $Y_1=y^{\perp}$. Then span$\{X_1,X_2\}=\RR^n$.
Choose $\{x_i\}_{i=1}^{n-1}$ vectors spanning $X_1$ and $\{x_i\}_{i=n}^{2n-2}$ be vectors spanning $X_2$. 
Then $\{x_i\}_{i=1}^{2n-2}$ is a frame for
$\RR^n$ with $x\perp x_i$, for $i=1,2,\ldots,n-1$ and $y\perp x_i$, for all $i=n,n+1,,\ldots,2n-2$. It follows that
\[ |\langle x+y,x_i\rangle| = |\langle x-y,x_i\rangle|,\mbox{ for all }i=1,2,\ldots,n,\]
but $x,y$ do not have the same or opposite signs and so $\{x_i\}_{i=1}^{2n-2}$ fails weak phase retrieval.

\begin{definition}
If $X$ is a subspace of $\RR^n$, we define the {\bf sphere of X} as
\[ S_X=\{x\in X:\|x\|=1\}.\]
\end{definition}

\begin{definition}
If $X,Y$ are subspaces of $\RR^n$, we define the {\bf distance between X and Y} as
\[ d(X,Y)=sup_{x\in S_X}inf_{y\in S_Y}\|x-y\|.\]
\end{definition}

It follows that if $d(X,Y)<\epsilon$ then for any $x\in X$ there is a $z\in S_Y$ so that $\|\frac{x}{\|x\|}-z\|<\epsilon$.
Letting $y=\|x\|z$ we have that $\|y\|=\|x\|$ and $\|x-y\|<\epsilon \|x\|$.

\begin{proposition}\label{P1}
Let $X,Y$ be hyperplanes in $\RR^n$ and unit vectors $x\perp X,\ y\perp Y$. If $d(X,Y)<\epsilon$ then $min\{\|x-y\|,\|x+y\|\}<6\epsilon.$
\end{proposition}

\begin{proof}
Since span$\{y,Y\}=\RR^n$, $x=ay+z$ for some $z\in Y$. By replacing $y$ by $-y$ if necessary, we may assume $0<a$.
By assumption, there is some $w\in X$ with $\|w\|=\|z\|$ so that $\|w-z\|< \epsilon$. Now
\[a = a\|y\|=\|ay\|=\|x-z\|\ge \|x-w\|-\|w-z\| \ge \|x\|-\epsilon=1-\epsilon.\]
So, $1-a< \epsilon$. Also, $1=\|x\|^2=a^2+\|w\|^2$ implies $a<1$. I.e. $0<1-a<\epsilon$.
\[1=\|x\|^2=\|ay+z\|^2= a^2\|y\|^2+\|z\|^2=a^2+\|z\|^2\ge (1-\epsilon)^2+\|z\|^2.\]
So
\[ \|z\|^2\le 1-(1-\epsilon)^2=2\epsilon-\epsilon^2\le 2\epsilon.\]
Finally,
\begin{align*}
 \|x-y\|^2&=\|(ay+z)-y\|^2\\
 &\le (\|(1-a)y\|+\|z\|)^2\\
 &\le (1-a)^2\|y\|^2+\|z\|^2+2(1-a)\|y\|\|z\|\\
 &< \epsilon^2+2\epsilon+2 \sqrt{2}\epsilon^2\\
 &<6\epsilon.
\end{align*}
\end{proof}

\begin{lemma}\label{L1}
Let $X,Y$ be hyperplanes in $\RR^n$, $\{x_i\}_{i=1}^{n-1}$ be a unit norm basis for $X$ and
$\{y_i\}_{i=1}^{n-1}$ be a unit norm basis for $Y$ with basis bounds B. If
$\sum_{i=1}^{n-1}\|x_i-y_i\|<\epsilon$ then $d(X,Y)<2\epsilon B$.
\end{lemma}

\begin{proof}
Let $0<A\le B < \infty$ be upper and lower basis bounds for the two bases.
Given a unit vector $x=\sum_{i=1}^{n-1}a_ix_i\in X$, let $y=\sum_{i=1}^{n-1}a_iy_i\in Y$.
We have that $sup_{1\le i \le n-1}|a_i|\le B$.
We compute:
\begin{align*}
 \|x-y\|&=\|\sum_{i=1}^{n-1}a_i(x_i-y_i)\|\\
 &\le \sum_{i=1}^{n-1}|a_i|\|x_i-y_i\|\\
 &\le (sup_{1\le i \le n-1}|a_i|)\sum_{i=1}^{n-1}\|x_i-y_i\|\le B\epsilon.
 \end{align*}
 So
 \[\|y\|\ge \|x\|-\|x-y\|\ge 1-B\epsilon.
 \]

\begin{align*}
\left \|x-\frac{y}{\|y\|}\right \|&\le  \|x-y\|+\left \|y-\frac{y}{\|y\|}\right \|\\
&\le B\epsilon+\frac{1}{\|y\|}\|(1-\|y\|)y\|\\
&= B\epsilon+(1-\|y\|)\\
&\le 2B\epsilon.
\end{align*}
It follows that $d(X,Y)<2B\epsilon.$
\end{proof}

\begin{lemma}\label{L3}
Let $\{x_i\}_{i=1}^n$ be a basis for $\RR^n$ with unconditional basis constant $B$ and assume $y_i\in \RR^n$ satisfies
$\sum_{i=1}^n\|x_i-y_i\|<\epsilon.$ Then $\{y_i\}_{i=1}^n$ is a basis for $\RR^n$ which is $1+\epsilon B$-equivalent
to $\{x_i\}_{i=1}^n$ and has unconditional basis constant $B(1+\epsilon B)^2$.
\end{lemma}

\begin{proof}
Fix $\{a_i\}_{i=1}^n$ and compute
\begin{align*}
\|\sum_{i=1}^na_iy_i\|&\le \|\sum_{i=1}^na_ix_i\|+\|\sum_{i=1}^n|a_i|(x_i-y_i)\|\\
&\le \|\sum_{i=1}^na_ix_i\|+(sup_{1\le i \le n}|a_i|)\sum_{i=1}^n\|x_i-y_i\|\\
&\le \|\sum_{i=1}^na_ix_i\|+(sup_{1\le i \le n}|a_i|)\epsilon\\
&\le \|\sum_{i=1}^na_ix_i\|+\epsilon B\|\sum_{i=1}^na_ix_i\|\\
& = (1+\epsilon B)\|\sum_{i=1}^na_ix_i\|.
\end{align*}
Similarly,
\[ \|\sum_{i=1}^n|a_i|y_i\|\ge (1-\epsilon B)\|\sum_{i=1}^na_ix_i\|.\]
So $\{x_i\}_{i=1}^n$ is $(1+\epsilon B)$-equivalent to $\{y_i\}_{i=1}^n$.

For $\epsilon_i = \pm1$,
\begin{align*}
 \|\sum_{i=1}^n\epsilon_ia_iy_i\|&\le (1+\epsilon B)\|\sum_{i=1}^n\epsilon_ia_ix_i\|\\
 &\le B(1+\epsilon B)\|\sum_{i=1}^na_ix_i\|\\
 &\le B(1+\epsilon B)^2\|\sum_{i=1}^na_iy_i\|.
\end{align*}

 and so $\{y_i\}_{i=1}^n$ is a $B(1+\epsilon B)$
unconditional basis.
\end{proof}

\begin{theorem}
The family of m-element weak phase retrieval frames are not dense in the set of m-element frames in $\RR^n$ for all $m\ge 2n-2$.
\end{theorem}

\begin{proof}
We may assume $m=2n-2$ since for larger m we just repeat the (2n-2) vectors over and over until we get m vectors.
Let $\{e_i\}_{i=1}^n$ be the canonical orthonormal basis for $\RR^n$ and let
$x_i=e_i$ for $i=1,2,\ldots,n$. By \cite{CCPW13}, there is an orthonormal sequence $\{x_i\}_{i=n+1}^{2n-2}$
so that $\{x_i\}_{i=1}^{2n-2}$ is full spark. Let $I=\{1,2,\ldots,n-1\}$.
Let $X=span\{x_i\}_{i=1}^{n-1}$ and $Y=span\{x_i\}_{i=n}^{2n-2}$. Then $x=e_n\perp X$ and there is a 
$\|y\|=1$ with $y\perp Y$. Note that $\langle x-y,e_n\rangle \not= 0 \not= \langle x+y,e_n\rangle$, for
otherwise, $x=\pm y\perp span \{x_i\}_{i\not= n}$, contradicting the fact that the vectors are full spark.  So there is a $j=n$ and a $\delta>0$
so that $|(x+y)(j)|,|(x-y)(j)|\ge \delta.$
We will show that there exists an $0<\epsilon$ so that whenever $\{y_i\}_{i=1}^{2n-2}$ are
vectors in $\RR^n$ satisfying
$ \sum_{i=1}^n\|x_i-y_i\|<\epsilon,$
then $\{y_i\}_{i=1}^n$ fails weak phase retrieval.

 Fix $0<\epsilon$. Assume $\{y_i\}_{i=1}^{2n-2}$ are vectors so that $\sum_{i=1}^{2n-2}\|x_i-y_i\|<\epsilon.$
Choose unit vectors $x'\perp span\{y_i\}_{i\in I},y'\perp span\{y_i\}_{i\in I^c}$. By Proposition \ref{P1} and Lemma \ref{L1},
we may choose $\epsilon$ small enough (and change signs if necessary) so that $\|x-x'\|,\|y-y'\|<\frac{\delta}{4B}$. Hence, since the
unconditional basis constant is $B$,
\begin{eqnarray*}
 |[(x+y)-(x'+y')](j)|&\le |(x-x')j|+|(y-y')(j)|\\&<B\|x-x'\|+B\|y-y'\|\\
&\le 2B\frac{\delta}{4B}=\frac{\delta}{2}.
 \end{eqnarray*}
It follows that
\[ |(x'+y')(j)|\ge |(x+y)(j)|-|[(x+y)-(x'+y')](j)|\ge \delta-\frac{1}{2}\delta=\frac{\delta}{2}.\]
Similarly, $|(x'-y')(j)|> \frac{\delta}{2}$. So $x'+y',x'-y'$ are not disjointly supported and so $\{y_i\}_{i=1}^{2n-2}$ fails weak phase retrieval
by Theorem \ref{T52}.
\end{proof}

\section{Classifying Weak Phase Retrieval}

In this section we will give several surprising equivalences and consequences of weak phase retrieval. These results give a complete understanding
of the difference between weak phase retrieval and phase retrieval.

We need a result from \cite{CAR2021}.

\begin{theorem}\label{T70}
Let $x=(a_1,a_2,\ldots,a_n),\ y=(b_1,b_2,\ldots,b_n)\in \RR^n$. If there exists an $i\in [n]$ so that $a_ib_i\not= 0$ and
$\langle x,y\rangle =0$, then $x,y$ do not weakly have the same signs or opposite signs. 
\end{theorem}

Now we give a surprising and very strong classification of weak phase retrieval. Theorem 13 in \cite{CAR2021}
is false. The following is what the result should have stated.

\begin{theorem}\label{T33}
Let $\{x_i\}_{i=1}^{m}$ be non-zero vectors in $\RR^n$. The following are equivalent:
\begin{enumerate}
\item The family $\{x_i\}_{i=1}^{m}$ does weak phase retrieval in $\RR^n$. 
\item If $x,y\in \RR^n$ and
\begin{equation}\label{E8} |\langle x,x_i\rangle|=|\langle y,x_i\rangle|\mbox{ for all } i=1,2,\ldots,m,
\end{equation}
then there is some $a\not= 0$ and a partition $\{I_j\}_{j=1}^5$ of $[1,n]$ (where any $I_j$ may
be empty) and 
\begin{enumerate}
\item For $i\in I_1$, $x(i)\not= 0$ and $y(i)=0$.
\item For $i\in I_2$, $x(i)=0$ and $y(i)\not= 0$.
\item For $i\in I_3$, $x(i)=0=y(i)$.
\item For $I\in I_4$, $x(i)=ay(i)$.
\item For $x\in I_5$, $x(i)=\frac{1}{a}y(i)$.
\end{enumerate}
\end{enumerate}
\end{theorem}

\begin{proof} $(2)\Rightarrow (1)$: Given (2), it is clear that x,y weakly have the same phase.
 \vskip7pt
 $(1)\Rightarrow (2)$: Assume $|\langle x,x_i\rangle|=|\langle y,x_i\rangle|$ for all
 $1\le i \le m$. We will examine cases:
 \vskip7pt
 \noindent {\bf Case 1}: $x=\pm y$. In this case, let $I_1,I_2$ be empty sets and a=1 and (2) of
 the theorem holds.

 \vskip7pt
 
 \noindent {\bf Case 2}: $\|x+y\|=\|x-y\|$. In this case,
\[
 \|x+y\|^2= \|x\|^2+\|y\|^2+2\langle x,y\rangle
 = \|x-y\|^2
 = \|x\|^2+\|y\|^2-2\langle x,y\rangle.
\]
It follows that $\langle x,y\rangle =0$. Since $\{x_i\}_{i=1}^m$ does weak phase retrieval, by Theorem \ref{T70},
$x(j)y(j)=0$ for all $j=1,2,\ldots,m$. So letting $I_4,I_5$ be empty sets, (2) of the theorem holds in this case. 
 \vskip7pt
 
 \noindent {\bf Case 3}: $\|x+y\|\not= 0 \not= \|x-y\|$ and 
 \[ \frac{1}{\|x+y\|}-\frac{1}{\|x-y\|} \not= 0.\]
 Now define $\{I_j\}_{j=1}^3$ as in (2) of the theorem, and define:
 \[ I=\{1\le i\le m:\langle x,x_i\rangle =\langle y,x_i\rangle\}\mbox{ so }I^c=\{1\le i\le m:\langle x,x_i\rangle = -\langle y,x_i\rangle\}.\]
 That is:
 \[ x-y\perp x_i\mbox{ for all }i\in I \mbox{ and }x+y\perp x_i\mbox{ for all }i\in I^c.\]
 By Theorem \ref{T52},
 \[ \frac{x+y}{\|x+y\|}+\frac{x-y}{\|x-y\|}\mbox{ and } \frac{x+y}{\|x+y\|}-\frac{x-y}{\|x-y\|}\mbox{ are disjointly supported}.\]
So for every $i=1,2,\ldots,2n-2$,
\[ \mbox{either }(\frac{x+y}{\|x+y\|}+\frac{x-y}{\|x-y\|})(i)=0\mbox{ or } (\frac{x+y}{\|x+y\|}-\frac{x-y}{\|x-y\|})(i)=0.\]
If
\[ (\frac{x+y}{\|x+y\|}+\frac{x-y}{\|x-y\|})(i)=0
\]
then
\[ \left ( \frac{1}{\|x+y\|}+\frac{1}{\|x-y\|}\right )x(i)=\left ( \frac{-1}{\|x+y\|}+\frac{1}{\|x-y\|}\right )y(i)
\]
Let
\[ a=\frac{\left ( \frac{-1}{\|x+y\|}+\frac{1}{\|x-y\|}\right )}{\left ( \frac{1}{\|x+y\|}+\frac{1}{\|x-y\|}\right )}.\]
Then $x(i)=ay(i)$ and by assumption, $a\not= 0$. In the second case:
\[ (\frac{x+y}{\|x+y\|}-\frac{x-y}{\|x-y\|})(i)=0,\]
and so
\[ \left ( \frac{1}{\|x+y\|}-\frac{1}{\|x-y\|}\right )x(i)=-\left ( \frac{1}{\|x+y\|}+\frac{1}{\|x-y\|}\right )y(i)
\]
So
\[ x(i) = \frac{-\left ( \frac{1}{\|x+y\|}+\frac{1}{\|x-y\|}\right )}{\left ( \frac{1}{\|x+y\|}-\frac{1}{\|x-y\|}\right )}y(i)=\frac{1}{a}y(i).\]
This proves (2)

\end{proof}

\begin{remark}
The above theorem shows exactly how weak phase retrieval and phase retrieval differ. In particular,
$\{x_i\}_{i=1}^m$ does phase retrieval if and only if whenever $|\langle x,x_i\rangle|=|\langle y,x_i\rangle|$,
for all $i=1,2,\ldots,m$ then $I_1,I_2$ are empty sets and $a=\pm 1$.
\end{remark}

We will look at some examples for the theorem. Let $\{x_i\}_{i=1}^4$ be the row vectors of the matrix in
Remark \ref{s18} which does weak phase retrieval. Let $x=(1,0,0),\ y=(0,1,0$. Then
\[ |\langle x,x_i\rangle|=|\langle y,x_i\rangle|,\mbox{ for all }i=1,2,\ldots,m.\]
Letting 
\[ I_1=\{1\}\ \ I_2=\{2\}\ \ I_3=\{3\} \ \ \mbox{ and }I_4,I_5\mbox{ are empty sets},\]
(2) of the theorem holds.

If $x=(2,3,0),\ y=(3,2,0)$ then $|\langle x,x_i\rangle|=\langle y,x_i\rangle|$, for
all $i=1,2,3,4.$ We obtain (2) of the theorem by letting $\{I_j\}_{j=1}^2$ be empty sets, $I_3=\{3\}$, $I_4=\{1\}$ 
and $I_5=\{2\}$and
$a=\frac{2}{3}$. Then
\[ x(1)=2=\frac{2}{3}3=ay(1)\mbox{ and }x(2)=3=\frac{3}{2}2=\frac{1}{a}y(2),\]
and(2) of the theorem holds.

\begin{definition}
Let $\{e_i\}_{i=1}^n$ be the canonical orthonormal basis of $\RR^n$. If $J\subset [n]$, we define
$P_J$ as the projection onto span$\{e_i\}_{i\in J}$.
\end{definition}

\begin{theorem}
Let $\{x_i\}_{i=1}^{m}$ be unit vectors in $\RR^n$. The following are equivalent:
\begin{enumerate}
\item Whenever $I\subset [2n-2]$ and $0\not= x\perp x_i$ for $i\in I$ and $0\not= y\perp x_i$ for $i\in I^c$,
there is no $j\in [n]$ so that $\langle x,e_j\rangle = 0 = \langle y,e_j\rangle$.
\item For every $J\subset [n]$ with $|J|=n-1$, $\{P_jx_i\}_{i=1}^{2n-2}$ does phase retrieval.
\item For every $J\subset [n]$ with $|J|<n$, $\{P_Jx_i\}_{i=1}^{2n-2}$ does phase retrieval.
\end{enumerate}
\end{theorem}

\begin{proof}
$(1)\Rightarrow (2)$: We prove the contrapositive. So assume (2) fails. Then choose $J\subset [n]$
with $|J|=n-1$, $J=[n]\setminus \{j\}$, and $\{P_Jx_i\}_{i=1}^{2n-2}$ fails phase retrieval. In particular, it fails complement property.
That is, there exists $I\subset [2n-2]$ and span $\{P_Jx_i\}_{i\in I}\not= P_J\RR^n$ and span $\{P_jx_i\}_{i\in I^c}
\not= P_J\RR^n$. So there exists norm one vectors $x,y$ in $P_J\RR^n$ with $P_Jx=x\perp P_Jx_i$ for all $i\in I$
and $P_Jy=y\perp P_Jx_i$ for all $i\in I^c$. Extend $x,y$ to all of $\RR^n$ by setting
$x(j)=y(j)=0$. Hence, $x\perp x_i$ for $i\in I$ and $y\perp x_i$ for $i\in I^c$, proving (1) fails.
\vskip7pt
$(2)\Rightarrow (3)$: This follows from the fact that every projection of a set of vectors doing phase retrieval
onto a subset of the basis also does phase retrieval.
\vskip7pt
$(3)\Rightarrow (2)$: This is obvious.
\vskip7pt
$(3)\Rightarrow (1)$: We prove the contrapositive. So assume (1) fails. Then there is a 
$I\subset [2n-2]$ and $0\not= x\perp x_i$ for $i\in I$ and $0\not= y\perp x_i$ for $i\in I^c$ and a
$j\in [n]$ so that
$\langle x,e_j\rangle =\langle y,e_j\rangle=0$. 
It follows that $x=P_Jx,\ y=P_Jy$ are non zero and $x\perp P_jx_i$ for all $i\in I$ and $y\perp P_jx_i$ for $i\in I^c$,
so $\{P_Jx_i\}_{i=1}^{2n-2}$ fails phase retrieval.
\end{proof}

\begin{remark}
The assumptions in the theorem are necessary. That is, in general, $\{x_i\}_{i=1}^{m}$ can do weak phase retrieval
and $\{P_Jx_i\}_{i=1}^m$ may fail phase retrieval. For example, in $\RR^3$ consider the row vectors $\{x_i\}_{i=1}^4$ of:
\[\begin{bmatrix}
1&1&1\\
-1&1&1\\
1&-1&1\\
1&1&-1
\end{bmatrix}\]
This set does weak phase retrieval, but if $J=\{2,3\}$ then $x=(0,1,-1) \perp P_Jx_i$ for $i=1,2$ and
$y=(0,1,1)\perp x_i$ for $i=3,4$ and $\{P_Jx_i\}_{i=1}^4$ fails phase retrieval.
\end{remark}

\begin{corollary}
Assume $\{x_i\}_{i=1}^{2n-2}$ does weak phase retrieval in $\RR^n$ and for every $J\subset [n]$
$\{P_Jx_i\}_{i=1}^{2n-2}$ does phase retrieval. Then if $x,y\in \RR^n$ and
\[|\langle x,x_i\rangle|=|\langle y,x_i\rangle|\mbox{ for all }i=1,2,\ldots,2n-2,\]
then there is a $J\subset [n]$ so that 
\[ x(j) = \begin{cases} a_j\not= 0\ \ for\ j\in J\\
0\ \ for \ j\in J^c
\end{cases}
 y(j) = \begin{cases}  0\ \ for\ j\in J\\
b_j\not= 0\ \ for \ j\in J^c
\end{cases}
\]
\end{corollary}

\begin{proposition}\label{P3}
Let $\{e_i\}_{i=1}^n$ be the unit vector basis of $\RR^n$ and for $I\subset [n]$, let $P_I$ be
the projection onto $X_I=span\{e_i\}_{i\in I}$.
 For every $m\ge 1$, there are vectors $\{x_i\}_{i=1}^m$ so
that for every $I\subset [1,n]$, $\{P_Ix_i\}_{i=1}^m$ is full spark in $X_I$.
\end{proposition}

\begin{proof}
We do this by induction on m. For m=1, let $x_1=(1,1,1,\ldots,1)$. This satisfies the theorem. So assume
the theorem holds for $\{x_i\}_{i=1}^m$. Choose $I\subset [1,n]$ with $|I|=k$. Choose $J\subset I$ with
$|J|=k-1$ and let $X_J=span\{x_i\}_{i\in J}\cup \{x_i\}_{i\in I^c}$. Then $X_J$ is a hyperplane in $\RR^n$ for
every $J$. Since there only exist finitely many such $J's$ there is a vector $x_{m+1}\notin X_J$ for every $J$.
We will show that $\{x_i\}_{i=1}^{m+1}$ satisfies the theorem.

Let $I\subset [1,n]$ and $J\subset I$ with $|J|=|I|$. If $P_Ix_{m+1} \notin X_J$, then $\{P_Ix_i\}_{i\in J}$ is linearly
independent by the induction hypothesis. On the other hand, if $m+1\in J$ then $x_{m+1}\notin X_J$. But,
if $P_Ix_{m+1}\in span \{P_Ix_i\}_{i\in J\setminus m+1}$, since $(I-P_I)x_{m+1}\in span \{e_i\}_{i\in I^c}$, it follows
that $x_{m+1}\in X_J$, which is a contradiction.
\end{proof}
\begin{remark}
In the above proposition, none of the $x_i$ can have a zero coordinate. Since if it does, projecting the vectors
onto that coordinate produces a zero vector and so is not full spark.
\end{remark}

\end{document}